\documentclass{amsart}

\usepackage{begnac}

\def\Fix{\operatorname{Fix}}

\DeclareMathOperator*{\prodB}{\prod\nolimits^{\mathrm{B}}}
\DeclareMathOperator*{\prodc}{\prod\nolimits^{\mathrm{c}}}
\DeclareMathOperator*{\prodec}{\prod\nolimits^{\mathrm{ec}}}

\begin{document}

\title{Unitary representations of locally compact groups as metric structures}

\author{Itaï \textsc{Ben Yaacov}}

\address{Itaï \textsc{Ben Yaacov},
  Univ Lyon,
  Université Claude Bernard Lyon 1,
  Institut Camille Jordan, CNRS UMR 5208,
  43 boulevard du 11 novembre 1918,
  69622 Villeurbanne Cedex,
  France}

\urladdr{\url{http://math.univ-lyon1.fr/~begnac/}}

\author{Isaac \textsc{Goldbring}}

\address{Isaac Goldbring, Department of Mathematics\\University of California, Irvine, 340 Rowland Hall (Bldg.\# 400),
  Irvine, CA 92697-3875, USA}
\email{\url{isaac@math.uci.edu}}
\urladdr{\url{http://www.math.uci.edu/~isaac}}

\thanks{The first author is partially supported by ANR grant AGRUME (ANR-17-CE40-0026).}
\thanks{The second author is partially supported by NSF grant DMS-2504477.}

\svnId $Id: LCGAction.tex 4793 2021-11-03 14:27:26Z begnac $
\thanks{\textit{Revision} {\svnRevision} \textit{of} \svnDate}

% \date{\today}
% \keywords{}
% \subjclass[2010]{}

\begin{abstract}
  For a locally compact group $G$, we show that it is possible to present the class of continuous unitary representations of $G$ as an elementary class of metric structures, in the sense of continuous logic.
  More precisely, we show how non-degenerate $*$-representations of a general $*$-algebra $A$ (with some mild assumptions) can be viewed as an elementary class, in a many-sorted language, and use the correspondence between continuous unitary representations of $G$ and non-degenerate $*$-representations of $L^1(G)$.

  We relate the notion of ultraproduct of logical structures, under this presentation, with other notions of ultraproduct of representations appearing in the literature, and characterise property (T) for $G$ in terms of the definability of the sets of fixed points of $L^1$ functions on $G$.
\end{abstract}

\maketitle

\tableofcontents

\section*{Introduction}

When suggesting a model-theoretic treatment of a mathematical object, or of a class of such objects, one must first present the said object(s) as logical structures.
In other words, to each of the objects in question we associate a logical structure, from which the original object can be recovered.

In some cases, such as fields or groups, this step is so straightforward that it is hardly noticed.
In others, such as valued fields, there is some (very small) degree of freedom, and one would say they can be viewed as structures in this language, or in that.
When dealing with metric structures in the formalism of continuous logic, a new difficulty arises, namely that one might wish to consider an unbounded metric space, such as a Banach space, in a logic which can only consider bounded metric spaces.

We know of two potential remedies for this.
First, one can sometimes extend the logic to one which does allow unbounded structures, with some price to pay at the level of technical complexity.
Second, one can sometimes argue that the unbounded structure can in fact be represented by a bounded structure (possibly many-sorted).
The second solution applies quite frequently in the case of Banach space structures.
Indeed, one can recover the entire Banach space from its unit ball equipped with the structure of a convex space.
This practice of restricting the domain of quantification to the unit ball is quite standard in other contexts -- for example, when defining the operator norm.

Whether such a presentation is ``correct'' is very context dependent, of course, but usually there are two or three good indicators for that, especially when dealing with a natural class of objects.
Either of the second or third conditions implies the first, and all three are ``usually'' satisfied (or not) simultaneously.
\begin{enumerate}
\item The class of (presentations of) objects should be closed under logical ultraproducts.
\item The class of (presentations of) objects should be elementary.
\item When there already exists a useful intrinsic notion of ultraproduct in the class, it should agree with the logical one.
\end{enumerate}

The unit ball of a Banach space is a good example of the positive case, where the intrinsic ultraproduct of Banach spaces (Dacunha-Castelle and Krivine \cite{DacunhaCastelle-Krivine:Ultraproduits}) coincides with the logical ultraproduct of their unit balls.

A failure of a logical presentation to satisfy these criteria often comes in one of two flavours, which we refer to as \emph{missing points} and \emph{extraneous points}, respectively.
As an example, let us consider two natural yet misguided attempts to represent the class of continuous unitary representations of a fixed non-discrete topological group $G$.
\begin{itemize}
\item A first naïve approach would be to consider the structure consisting of the unit ball of a Hilbert space, together with a function symbol for the action of each $g \in G$.
  While the map $g \mapsto g\xi$ is continuous (at $1 \in G$) for each $\xi$ in the structure, the family of such maps is not necessarily equicontinuous, and a logical ultrapower may well contain $\xi$ such that $g \mapsto g\xi$ is not continuous at all.
  This is an ``extraneous point'' (one kind of non-logical ultraproduct of continuous representations is defined exactly by excluding such points from the Banach space ultraproduct).
\item A next attempt might be to consider, for each modulus of continuity at $1$, the sort of all $\xi$ in the unit ball such that $g \mapsto g\xi$ satisfies that modulus of continuity.
  In such a structure, associated to a continuous representation, $G$ acts continuously, and even equicontinuously, on each sort, and this is preserved by ultraproducts.

  However, such a structure must also satisfy a subtler condition.
  Say $A$ and $B$ are two moduli of continuity, with $B$ stronger than $A$.
  Then the associated sorts must satisfy an inclusion relation $S_B \subseteq S_A$.
  Moreover, any point of $S_A$ that happens to satisfy the modulus $B$ must belong to $S_B$ -- and this property need not be preserved under ultraproducts.
  If this fails, then we may say that the sort $S_B$ is ``missing'' some points.
\end{itemize}

In the case of a locally compact $G$ we propose a solution, by splitting the representation into sorts not by moduli of continuity, but as images of the action of $L^1(G)$.
The resulting structure may be presented most elegantly as a non-degenerate representation of the $*$-algebra $L^1(G)$.
We therefore begin with a general discussion of the presentation of non-degenerate representations of $*$-algebras as logical structures in \autoref{sec:NonDegenerateStarRepresentations}.
How this specialises to representations of a locally compact $G$ is discussed in \autoref{sec:ContinuousUnitaryRepresentationsLocallyCompactGroups} (this is fairly standard and mostly included for the sake of completeness).
In \autoref{sec:Ultraproduct} we put the notion of ultraproduct associated with out logical structures in the context of notions of ultraproduct of unitary representations existing in the literature.
Finally, in \autoref{sec:PropertyT} we give a model-theoretic characterisation of Kazhdan's property (T) in $G$ in terms of definability of the set(s) of fixed points in the associated structure.

\section{Non-degenerate $*$-representations}
\label{sec:NonDegenerateStarRepresentations}

Throughout, by an \emph{algebra} we mean a complex algebra, i.e., a ring $A$, not necessarily commutative or unital, that is also a complex vector space, satisfying $\alpha(ab) = (\alpha a) b = a (\alpha b)$ for all $a,b \in A$ and $\alpha \in \bC$.
An algebra equipped with a semi-linear involution $*$ satisfying $(ab)^* = b^* a^*$ is a \emph{$*$-algebra}.
An algebra equipped with a norm satisfying $\|ab\| \leq \|a\|\|b\|$ is a \emph{normed algebra}, and a \emph{Banach algebra} if it is complete.
If it is both a normed (Banach) algebra and a $*$-algebra, and $\|a^*\| = \|a\|$, then it is a \emph{normed (Banach) $*$-algebra}.
A \emph{$C^*$-algebra} is a Banach $*$-algebra in which $\|a a^*\| = \|a\|^2$.

A \emph{morphism} of normed algebras is a bounded linear map that respects multiplication, and a \emph{$*$-morphism} of normed $*$-algebras is a normed algebra morphism that respects the involution.
A \emph{$*$-representation} of a $*$-algebra $A$ in a Hilbert space $E$ is a $*$-morphism $\pi\colon A \rightarrow B(E)$.

\begin{fct}
  \label{fct:StarMorphismContractive}
  Let $A$ be a Banach $*$-algebra and $B$ a $C^*$-algebra.
  Then any $*$-morphism $\varphi\colon A \rightarrow B$ is contractive.
  In particular, any $*$-representation of a Banach $*$-algebra is contractive.
\end{fct}
\begin{proof}
  See Folland~\cite[Proposition~1.24(b)]{Folland:AbstractHarmonicAnalysis}.
  The hypothesis that $B$ is unital is superfluous by \cite[Proposition~1.27]{Folland:AbstractHarmonicAnalysis}.
\end{proof}

Given a Banach $*$-algebra $A$, a $*$-representation of $A$ can be naturally viewed as a single-sorted metric structure.
Indeed, all we need to do is take the unit ball of a Hilbert space $E$, and for each $a \in A$ of norm at most one, name the operator $\pi(a)$ in the language.
The class of all such structures is elementary, defined by universal axioms (modulo the axioms for a unit ball of a Banach space), and if $A$ is separable, then by choosing a dense sub-family of $A$ the language can be made countable.

This is a little too easy, and falls short of what we want to achieve.
When $A$ is unital, it would be natural to add the requirement that $\pi(1) = \id$.
In the general case, this requirement can be replaced with non-degeneracy (see for example \cite[Section~4]{Cherix-Cowling-Straub:FilterProducts}).

\begin{dfn}
  \label{dfn:NondegenerateRepresentation}
  Let $A$ be a Banach algebra and let $\pi\colon A \rightarrow B(E)$ be a $*$-representation.
  The \emph{non-degenerate part} of the representation, let us call it $E^\pi$, is the closed subspace generated by all $\pi(a) \xi$ for $a \in A$ and $\xi \in E$.
  If $E^\pi = E$, then $\pi$ is \emph{non-degenerate}.
\end{dfn}

\begin{fct}
  \label{fct:NondegenerateRepresentation}
  Let $\pi\colon A \rightarrow B(E)$ be a $*$-representation of a Banach $*$-algebra $A$, and let $\xi \in E$.
  \begin{enumerate}
  \item We have $\xi \perp E^\pi$ if and only if $\pi(a) \xi = 0$ for every $a \in A$.
    In particular, the $*$-representation restricts to a non-degenerate one on $E^\pi$.
  \item Assume that $(e_\alpha)$ is a left approximate unit for $A$.
    Then $\xi \in E^\pi$ if and only if $\pi(e_\alpha) \xi \rightarrow \xi$.
    Equivalently (given the first item), for any $\xi \in E$, the sequence $\pi(e_\alpha)\xi$ converges to the orthogonal projection of $\xi$ to $E^\pi$.
  \end{enumerate}
\end{fct}
\begin{proof}
  The first item follows from the identity $\bigl\langle \xi, \pi(a) \zeta \bigr\rangle = \bigl\langle \pi(a^*) \xi, \zeta \bigr\rangle$.
  For the second, if $\pi(e_\alpha) \xi \rightarrow \xi$, then $\xi \in E^\pi$ by definition.
  For the converse, we may assume that $\xi$ is of the form $\pi(a) \zeta$.
  Then, by \autoref{fct:StarMorphismContractive}:
  \begin{gather*}
    \|\pi(e_\alpha) \xi - \xi\|
    = \|\pi(e_\alpha a) \zeta - \pi(a) \zeta\|
    \leq \| e_\alpha a - a \| \|\zeta\| \rightarrow 0.
    \qedhere
  \end{gather*}
\end{proof}

If $A$ does not have a unit, then the class of non-degenerate $*$-representations of $A$, presented naïvely as above, need not be elementary, so something better needs to be done.
The problem is that we cannot express an infinite disjunction such as ``there exists $a \in A$ such that $\xi$ is close to the image of $\pi(a)$''.
We solve this by using a many-sorted language: for each $a \in A$ we shall have a sort $S_a$, consisting of the closure of the image of the unit ball under $\pi(a)$, and all that is left is to express the interactions between these sorts.

In what follows, by a \emph{symmetric convex space}, we mean a complete convex subset of a Banach space closed under opposite.
Following \cite{BenYaacov:NakanoSpaces}, a (bounded) symmetric convex space will be considered as a metric structure in the language $\{0,-,\half[x+y]\}$, where $-$ is the unary opposite operation and $\half[x+y]$ is the binary average operation.
If $E$ is a real Banach space and $C \subseteq E$ is a symmetric convex set that generates a dense subset of $E$, then $E$ can be recovered from $C$.
Linear maps can be recovered using the following easy result.

\begin{lem}
  \label{lem:ConvexSpaceLinearMap}
  Let $E$ and $F$ be real normed spaces, $C \subseteq E$ a symmetric convex generating subset.
  Assume that $f\colon C \rightarrow F$ is bounded in the sense that $\|f(x)\| < \alpha \|x\|$ for some $\alpha \in \bR$.
  Then the following are equivalent:
  \begin{enumerate}
  \item The map $f$ respects the convex structure: $f(0) = 0$ and $f(\half[x+y]) = \half[f(x) + f(y)]$.
  \item The map $f$ is additive: $f(x+y) = f(x) + f(y)$ whenever $x,y,x+y \in C$.
  \item The map $f$ extends to a linear bounded map $E \rightarrow F$.
  \end{enumerate}
\end{lem}
\begin{proof}
  \begin{cycprf}
  \item For $x \in C$ we have $f(x/2) = f(\half[x+0]) = \half[f(x)+0] = f(x)/2$.
    Therefore, if $x,y,x+y \in C$, then
    \begin{gather*}
      f(x+y) = 2 f\left( \half[x+y] \right) = 2 \half[f(x)+f(y)] = f(x) + f(y).
    \end{gather*}
  \item If $f$ is additive on $C$, then it extends to an additive map $E \rightarrow F$.
    Such a map is necessarily $\bQ$-linear and bounded with the same constant $\alpha$, so it is $\bR$-linear.
  \item[\impfirst]
    Immediate.
  \end{cycprf}
\end{proof}

\begin{dfn}
  Let $A$ be a Banach $*$-algebra, and $\pi\colon A \rightarrow B(E)$ a non-degenerate $*$-representation.
  We associate to it a multi-sorted structure $M = M(E,\pi)$ constructed as follows.
  \begin{itemize}
  \item For each $a \in A$, $M$ admits a sort $S_a = \overline{\pi(a) E_{\leq 1}}$, where $E_{\leq 1}$ denotes the unit ball of $E$.
  \item Each sort is equipped with the structure of a symmetric convex space, as well as with a symbol for multiplication by $i$.
  \item For any $a,b \in A$, we name the restriction to $S_a \times S_b$ of the real part of the inner product: $[\xi,\zeta] = \Re \langle \xi,\zeta \rangle$.
    Since one such predicate exists for each pair of sorts, we may sometimes write $[\cdot,\cdot]_{a,b}$.
  \item For any $a,b \in A$, the map $\pi(a) \colon S_b \rightarrow S_{ab}$ is named by a function symbol $\pi_a$.
  \end{itemize}
\end{dfn}

All the symbols are bounded and uniformly continuous in a manner that does not depend on the choice of $(E,\pi)$, so these can all be viewed as structures in a common language, call it $\cL^A$.
Of course, the same construction applies even if $(E,\pi)$ has a degenerate part, but this degenerate part will not be reflected in any way in the structure $M(E,\pi)$.

We define $T^A$ to be the theory consisting of the following axiom schemes that we explain shortly.
All the axioms are either stated in continuous logic, or can easily be.
Axioms \autoref{ax:Conv} to \autoref{ax:Complex} are universal, with implicit universal quantifiers.

\begin{align*}
  \tag{Conv} \label{ax:Conv} & \text{Each $S_a$ is a symmetric convex space of radius $\leq \|a\|$,}
  \\
  \tag{Sym} \label{ax:Sym} & [\xi,\zeta] = [\zeta,\xi],
  \\
  \tag{Lin1} \label{ax:Lin1} & \!\! \left[ \xi, \half[\zeta+\zeta'] \right] = \half[{[\xi,\zeta] + [\xi,\zeta']}],
  \\
  \tag{Lin2} \label{ax:Lin2} & [ \xi, 0 ] = 0,
  \\
  \tag{Norm} \label{ax:Norm} & [\xi,\xi] = \|\xi\|^2,
  \\
  \tag{Pos} \label{ax:Pos} & \sum_{i,j=1}^n [\xi_i,\xi_j] \geq 0,
  \\
  \intertext{From here on, let $\|\sum \xi_i\|$ be short for $\sqrt{\sum [\xi_i,\xi_j]}$.}
  \tag{Pi1} \label{ax:Pi1} & \left\| \sum \pi_a \xi_i \right\| \leq \|a\| \|\sum \xi_i\|,
  \\
  \tag{Pi2} \label{ax:Pi2} & a \mapsto \pi_a \ \text{is a $*$-morphism},
  \\
  \tag{Complex} \label{ax:Complex} & i\colon S_a \rightarrow S_a \ \text{respects all other symbols,} \ i^2\xi = -\xi,
  \\
  \tag{BallImg} \label{ax:BallImg} & \sup_{\xi_i \in S_{b_i}}\inf_{\zeta \in S_a} \, \left\| \sum \pi_a \xi_i - \zeta \right\| \leq \left| \left\| \sum \xi_i \right\| - 1 \right| \|a\|,
  \\
  \tag{DenseImg} \label{ax:DenseImg} & \sup_{\xi \in S_{ab}} \ \inf_{\zeta \in S_b} \, \left\| \pi_a \zeta - \xi \right\| = 0,
  \\
  \tag{HausDist} \label{ax:HausDist} & \sup_{\xi \in S_a} \ \inf_{\zeta \in S_b} \, \left\| \zeta - \xi \right\| \leq \|a-b\|.
\end{align*}

Clearly, every structure of the form $M(E,\pi)$ is a model of $T^A$.
Now let $M$ be any model of $T^A$, or, at a first time, merely of \autoref{ax:Conv} to \autoref{ax:Complex}.

\begin{enumerate}
\item
  Axiom \autoref{ax:Conv} requires that each sort $S_a$ be a symmetric convex space of radius at most $\|a\|$.
  This is indeed expressible in continuous logic and the generated real normed space can be recovered, call it $E_a$ (see for example \cite{BenYaacov:NakanoSpaces}).
\item
  For each pair $a,b \in A$, axioms \autoref{ax:Sym} to \autoref{ax:Lin2} require $[\cdot,\cdot]$ on $S_a \times S_b$ to be symmetric and $\bR$-bilinear -- that it to say that it extends (uniquely) to an $\bR$-bilinear form on $E_a \times E_b$, as per \autoref{lem:ConvexSpaceLinearMap}.
  By axiom \autoref{ax:Norm}, it defines the norm on $E_a$.
\item
  Let $F^1 = \bigoplus_{a \in A} E_a$.
  If $\xi,\zeta \in F^1$, say $\xi = \sum \xi_i$ and $\zeta = \sum \zeta_i$ where $\xi_i,\zeta_i \in E_{a_i}$, let $[\xi,\zeta] = \sum_{i,j} [\xi_i,\zeta_j]$.
  This defines a symmetric $\bR$-bilinear form on $F^1$, and axiom \autoref{ax:Pos} requires it to be positive semi-definite.
  It follows that $\|\sum \xi_i\| = \sqrt{\sum [\xi_i,\xi_j]}$ is a semi-norm.
  Let $F^0 \subseteq F^1$ be its kernel.
  Then $F = \overline{F^1 / F^0}$ with the induced norm is a real Hilbert space.
\item
  Axiom \autoref{ax:Pi1} implies, first of all, that $\pi_a\colon S_b \rightarrow S_{ab} \subseteq E_{ab} \subseteq F$ is bounded:
  \begin{gather}
    \label{eq:PiBounded}\|
    \pi_a \xi\| \leq \|a\| \|\xi\|.
  \end{gather}
  The axiom also implies that $\pi_a\colon S_b \rightarrow F$ is additive, in the sense of \autoref{lem:ConvexSpaceLinearMap}, and therefore it extends uniquely to an $\bR$-linear map $E_b \rightarrow F$.
  These combine to a single $\bR$-linear map $\pi_a\colon F^1 \rightarrow F$, and axiom \autoref{ax:Pi1} implies that this combined map also satisfies \autoref{eq:PiBounded}.
  Therefore it induces an operator $\sigma(a) \in B(F)$ of norm at most $\|a\|$.
\item
  Now that we have a map $\sigma\colon A \rightarrow B(F)$, we require it to be a $*$-morphism (with respect to the real inner product $[\cdot,\cdot]$).
  This consists of a long list of identities, which we chose to omit (axiom \autoref{ax:Pi2}).
\item
  Axiom \autoref{ax:Complex} means that multiplication by $i$ is isometric and linear (\autoref{lem:ConvexSpaceLinearMap} again), putting a complex structure on $F$.
  Since $i$ commutes with each $\pi_a$, each $\sigma(a)$ is $\bC$-linear.
  Following the convention that a sesquilinear form is linear in the first argument and semi-linear in the second, we may recover a complex inner product by $\langle \xi,\zeta \rangle = [\xi,\zeta] + i [\xi,i \zeta]$.
\end{enumerate}

If $(E,\pi)$ is a non-degenerate $*$-representation and $M = M(E,\pi)$, then the isometric embeddings $S_a = \overline{\pi(a) E_{\leq 1}} \hookrightarrow F^1 \rightarrow F$ glue (by non-degeneracy) to a canonical isometric linear bijection $E \rightarrow F$, which is the desired isomorphism of $*$-representations $(E,\pi) \cong (F,\sigma)$ (if $(E,\pi)$ is degenerate, then we recover its non-degenerate part).

If $M$ is an arbitrary model of \autoref{ax:Conv} to \autoref{ax:Complex}, then we recover a $*$-representation $(F,\sigma)$.
However, $M(F,\sigma)$ need not be isomorphic to $M$, since we still need to say that $S_a \subseteq F$ is exactly $\overline{\sigma(\pi) F_{\leq 1}}$.
This will follow from the three last axioms, provided we make one additional hypothesis regarding $A$: that for every $a \in A$ there exists $b \in A$ of norm one, such that $ab$ is arbitrarily close to $a$.
This holds, in particular, if $A$ admits a right approximate identity $(e_\alpha)$ of norm one (in which case $(e_\alpha^*)$ is a left approximate identity, and $(e_\alpha + e_\alpha^* - e_\alpha e_\alpha^*)$ is a two-sided approximate identity, albeit not necessarily of norm one).

\begin{enumerate}
\item The inclusion $S_a \supseteq \overline{\sigma(\pi) F_{\leq 1}}$ is exactly axiom \autoref{ax:BallImg}.
\item For the opposite inclusion, let $a \in A$ and $\varepsilon > 0$, and choose $b \in A$ of norm one such that $\|a-ab\|< \varepsilon$.
  By axioms \autoref{ax:DenseImg} and \autoref{ax:HausDist}, for any $\xi \in S_a$ there exists $\zeta \in S_b$ such that $\|\pi_a \zeta - \xi\| < \varepsilon$, and $\|\zeta\| \leq \|b\| = 1$.
\end{enumerate}

Putting this all together, we have proven the following.

\begin{thm}
  \label{thm:NonDegenerateRepresentationClass}
  Let $A$ be a Banach $*$-algebra.
  Assume that $A$ admits an approximate one-sided identity of norm one, or merely that every $a \in A$ belong to $\overline{a A_{\leq 1}}$.

  Then the class of non-degenerate $*$-representations of $A$ can be identified with the class of models of $T^A$, and therefore may be considered to be an elementary class.
\end{thm}

In particular, if $A$ satisfies the hypotheses of \autoref{thm:NonDegenerateRepresentationClass}, then the class of non-degenerate $*$-representations of $A$ is closed under the ultraproduct/ultrapower construction applied to $\cL^A$-structures.
This coincides with the \emph{non-degenerate ultraproduct/ultrapower}, as proposed in \cite[Section~4]{Cherix-Cowling-Straub:FilterProducts}, obtained by taking the Banach space ultraproduct/ultrapower of $E$ (see \autoref{sec:Ultraproduct}), which is naturally a $*$-representation, and taking its non-degenerate part.

\section{Continuous unitary representations of locally compact groups}
\label{sec:ContinuousUnitaryRepresentationsLocallyCompactGroups}

Let $G$ be a topological group, and let $M(G)$ denote the space of regular complex Borel measures on $G$.
For every $\mu \in M(G)$ there exists a unique finite positive measure $|\mu|$ such that $d\mu = \alpha d|\mu|$, where $\alpha\colon G \rightarrow \bC$ is a Borel function into the unit circle.  We set the \emph{total variation} of $\mu$ to be $\|\mu\|=|\mu|(G)$.
For $\mu,\nu \in M(G)$ we may define an \emph{involution} and a \emph{convolution} by
\begin{gather*}
  \mu^*(g) = \overline{i_* \mu},
  \qquad
  \mu * \nu = m_* (\mu \otimes \nu),
\end{gather*}
where $m\colon G^2 \rightarrow G$ is the group law and $i\colon G \rightarrow G$ is inversion.
It is easy to check that $\mu^*$ and $\mu * \nu$ belong to $M(G)$, and
% \begin{gather*}
%   \|\mu^*\| = \|\mu\|, \qquad \|\mu * \nu\| \leq \|\mu\| \|\nu\|.
% \end{gather*}
that equipped with these two operations and with the norm $\|\mu\|$, $M(G)$ is a Banach $*$-algebra.

For every $g \in G$ we have a Dirac measure $\delta_g \in M(G)$, and $\delta_e$ is the unit of $M(G)$.
The map $g \mapsto \delta_g$ is injective and respects the algebraic structure: $\delta_g^* = \delta_{g^{-1}}$, $\delta_g * \delta_h = \delta_{gh}$.
(However, the norm topology on $M(G)$ induces the discrete topology on $G$ under this identification.)
In particular, $G$ acts isometrically on $M(G)$ on either side by convolution with the Dirac measure:
\begin{gather*}
  f \mu h = \delta_f * \mu * \delta_h, \qquad \|f \mu h\| = \|\mu\|.
\end{gather*}

A \emph{continuous (always unitary) representation} of $G$ consists of a Hilbert space $E$ equipped with a continuous unitary action $G \curvearrowright E$, or equivalently, with a continuous morphism $\pi\colon G \rightarrow U(E)$, where $U(E)$ is equipped with the strong (equivalently, weak) operator topology.
This means in particular that we may write $g \xi$ and $\pi(g) \xi$ interchangeably, when $g \in G$ and $\xi \in E$.
For a unitary action $G \curvearrowright E$ to be (jointly) continuous it suffices that for each $\xi \in E$ (separately), the map $g \mapsto g\xi$ be continuous at the identity.

Any continuous representation of $G$ in $E$ can be extended naturally to a $*$-representation of $M(G)$, by
\begin{gather*}
  \pi(\mu) = \int_G \pi(g) d\mu(g).
\end{gather*}
By this we mean that $\pi(\mu) \xi \in E$ is the unique vector such that
\begin{gather}
  \label{eq:RepresentationMG}
  \bigl\langle \pi(\mu) \xi, \zeta \bigr\rangle = \int_G \langle g\xi, \zeta \rangle \, d\mu(g).
\end{gather}
We have $\|\pi(\mu)\| \leq \|\mu\|$ by Cauchy-Schwarz, giving rise to a map $\pi\colon M(G) \rightarrow B(E)$.
It is easy to check that $\pi(\mu^*) = \pi(\mu)^*$ and $\pi(\mu * \nu) = \pi(\mu) \pi(\nu)$, so $\pi$ is indeed a $*$-representation.
In addition, $\pi(\delta_g) = \pi(g)$, so this representation extends the original one via our identification $G \subseteq M(G)$.

\begin{lem}\label{lem:ContinuousRepresentationMG}
  Let $\xi \in E$, and let $\mu \in M(G)$ be a probability measure concentrated on the set $\bigl\{g \in G : \|g \xi - \xi\| \leq r \bigr\}$.
  Then $\|\pi(\mu) \xi - \xi\| \leq r$.
\end{lem}
\begin{proof}
  For all $\zeta \in E$ we have $\bigl| \langle \pi(\mu) \xi - \xi, \zeta \rangle \bigr| \leq r \|\zeta\|$ by a direct application of \autoref{eq:RepresentationMG} and the hypotheses.
\end{proof}

Let us add the hypothesis that $G$ is locally compact, and choose a left Haar measure $H$.
Then $H(f A h) = \Delta(h) H(A)$, where $\Delta \colon G \rightarrow (\bR^{>0},\cdot)$ is the modular function on $G$, a group morphism that only depends on $G$.
We shall write $dg$ for $dH(g)$.

We may identify $\varphi \in L^1(G)$ (with respect to $H$) with $\mu_\varphi \in M(G)$ defined by $d\mu_\varphi = \varphi \, dH$.
The map $\varphi \mapsto \mu_\varphi$ is a linear isometry, so $L^1(G) \subseteq M(G)$ is a Banach subspace.
If $\varphi,\psi \in L^1(G)$, then (under the identification of $\varphi$ with $\mu_\varphi$)
\begin{gather*}
  \varphi^*(g) = \Delta(g^{-1}) \overline{\varphi}(g^{-1})
\end{gather*}
and
\begin{gather*}
  (\psi * \varphi)(g)
  = \int_G \psi(h) \varphi(h^{-1} g) \, dh
  = \int_G \Delta(h^{-1}) \psi(g h^{-1}) \varphi(h) \, dh.
\end{gather*}
More generally, for any $\varphi \in L^1(G)$ and $\mu \in M(G)$:
\begin{gather*}
  (\mu * \varphi)(g)
  = \int_G \varphi(h^{-1} g) \, d\mu(h),
  \qquad
  (\varphi * \mu)(g)
  = \int_G \Delta(h^{-1}) \varphi(g h^{-1}) \, d\mu(h),
\end{gather*}
so
\begin{gather*}
  (\delta_f * \varphi * \delta_h)(g) = \Delta(h^{-1}) \varphi(f^{-1} g h^{-1}).
\end{gather*}

In particular, $L^1(G) \subseteq M(G)$ is a $*$-subalgebra, so to every continuous representation $(E,\pi)$ of $G$ corresponds a canonical $*$-representation of $L^1(G)$, through the restriction of the $*$-representation of $M(G)$:
\begin{gather*}
  \pi(\varphi) = \int_G \varphi(g) \pi(g) \, dg.
\end{gather*}
The usefulness of this representation is due to the following classical fact:
\begin{fct}\label{fct:ContinuousAction}
  Let $\varphi \in L^1(G)$.
  Then $\delta_f * \varphi * \delta_h \rightarrow \varphi$ in $L^1(G)$ as $f,h \rightarrow e$.
\end{fct}
\begin{proof}
  When $\varphi$ is bounded, this follows from dominated convergence, and the general case follows by a density argument.
\end{proof}

Let $(U_\alpha)$ be a basis of compact neighbourhoods of $1$, and let $\alpha \leq \beta$ when $U_\alpha \supseteq U_\beta$.
Let $e_\alpha \in C_c(G) \subseteq L^1(G)$ be continuous, positive, of norm one, supported in $U_\alpha$.
Then as a net, $(e_\alpha)$ is an approximate identity of $L^1(G)$, that is to say that $e_\alpha * \varphi \rightarrow \varphi$ and $\varphi * e_\alpha \rightarrow \varphi$ (in norm) for every $\varphi \in L^1(G)$.
If $(E,\pi)$ is a continuous representation of $G$, then $\pi(e_\alpha) \xi \rightarrow \xi$ for every $\xi \in E$ (e.g., by \autoref{lem:ContinuousRepresentationMG}).
In particular, $E^\pi = E$ and $\pi\colon L^1(G) \rightarrow B(E)$ is non-degenerate.

Conversely, let $\pi\colon L^1(G) \rightarrow B(E)$ be any non-degenerate $*$-representation.
For any $g \in G$, the map $\varphi \mapsto \delta_g * \varphi$ is isometric.
If $\xi \in E$, then $\pi(e_\alpha) \xi \rightarrow \xi$ by non-degeneracy, so $\pi(\delta_g * e_\alpha) \xi$ must converge as well, call its limit $\pi(g) \xi$ or $g\xi$.
This defines a group morphism $\pi\colon G \rightarrow U(E)$.
A combination of \autoref{fct:ContinuousAction} with the rate of convergence of $\pi(e_\alpha) \xi$ to $\xi$ yields that $g \mapsto g\xi$ is continuous at $e$ for any fixed $\xi \in E$, so $\pi\colon G \rightarrow U(E)$ is a continuous representation.

\begin{fct}
  These operations are one the inverse of the other, yielding a bijective correspondence between continuous representations of $G$ on $E$ and non-degenerate $*$-representations of $L^1(G)$ on $E$.
\end{fct}

\begin{proof}
  See Folland~\cite[Theorem 3.11]{Folland:AbstractHarmonicAnalysis}.
\end{proof}

Since $L^1(G)$ admits an approximate identity of norm one, the results of the previous section apply.
In other words, identifying a continuous representation of $G$ with the corresponding representation of $L^1(G)$, and the latter with the corresponding model of $T^{L^1(G)}$, we may view the class of continuous representations of $G$ as elementary.

If $G$ is discrete, a unitary representation of $G$ can also be considered as a single-sorted structure, consisting of the unit ball of a Hilbert space with each unitary operator $\pi(g)$ (restricted to the unit ball) named in the language (see for example by Berenstein \cite{Berenstein:HilbertWithAutomorphismGroups}).
In this case, the Haar measure is (a multiple of) the counting measure, and $\Delta \equiv 1$.
Identifying $g \in G$ with $\delta_g$ we have $G \subseteq L^1(G)$.
Viewing a representation of $G$ as an $\cL^{L^1(G)}$-structure, the sort associated to $\delta_g$ is the entire unit ball, for any $g \in G$, and $\pi(\delta_h)$ acts on it as $\pi(h)$ for all $h \in G$.
In other words, we may recover the single-sorted structure alluded to above as a reduct of the $\cL^{L^1(G)}$-structure.
Conversely, we can interpret the multi-sorted $\cL^{L^1(G)}$-structure in the single-sorted one.
The full details would involve more definitions than interesting results, so we shall omit them.

\section{Ultraproduct constructions}
\label{sec:Ultraproduct}

Let us discuss possible ultraproduct constructions for unitary actions.
Throughout this discussion, $I$ is a set and $\cU$ is an ultrafilter on $I$.
Let $\prodB_\cU E_i$ denote the Banach space ultraproduct of a family of Banach spaces, possibly with additional structure.

In particular, if $E = \prodB_\cU E_i$ and $C = \prodB_\cU B(E_i)$, then $C$ is again a $C^*$-algebra, and there is a natural isometric embedding of $C^*$-algebras:
\begin{gather*}
  \prodB_\cU B(E_i) \subseteq B(E),
\end{gather*}
where $[T_i] \in \prodB_\cU B(E_i)$ acts on $E$ by $[T_i] [\xi_i] = [T_i \xi_i]$.

If each $(E_i,\pi)$ is a unitary representation (not necessarily continuous) of $G$, then $(\prodB_\cU E_i,\pi_\cU)$ is again such a representation, where $\pi_\cU(g) = \bigl[ \pi_i(g) \bigr] \in \prodB_\cU B(E_i) \subseteq B(E)$.
We call this the \emph{naïve} ultraproduct.
Even if each $E_i$ is a continuous representation, the naïve ultraproduct need not be so.
In the literature one finds two main ideas for remedying this deficiency.

First, any unitary representation admits a continuous part:
\begin{dfn}
  \label{dfn:ContinuousPart}
  Let $E$ be a Hilbert space, and $G \rightarrow U(E)$ a unitary representation, not necessarily continuous.
  We define $E^{\mathrm{c}}$ to consist of all $\xi \in E$ for which the map $g \mapsto g \xi$ is continuous.
  We call it the \emph{continuous part} of the unitary representation $E$.

  In particular, if $(E_i)$ are unitary representations (say, continuous, but this is not required for this definition), we define the \emph{continuous ultraproduct} as
  \begin{gather*}
    \prodc_\cU E_i = \left( \prodB_\cU E_i\right)^{\mathrm{c}}.
  \end{gather*}
\end{dfn}

The following is easy:

\begin{fct}
  \label{fct:ContinuousPart}
  With the hypotheses of \autoref{dfn:ContinuousPart}, $E^{\mathrm{c}}$ is a Hilbert subspace of $E$.
  It is moreover $G$-invariant, and the restricted representation $G \rightarrow U(E^c)$ is continuous.

  In particular, the continuous ultraproduct of (continuous) unitary representations of $G$ is a continuous unitary representation.
\end{fct}

For a slightly different, \textit{a priori} stronger, approach, recall that a \emph{semi-norm} on $G$ is a function $\rho\colon G \rightarrow \bR^+$ satisfying $\rho(1) = 0$ and $\rho(g^{-1}f) \leq \rho(g) + \rho(f)$.
It is a \emph{norm} if $\rho(g) = 0$ implies $g = 1$.
Let us write $\{\rho < \varepsilon\}$ for $\bigl\{g \in G : \rho(g) < \varepsilon\bigr\}$.
A semi-norm is continuous if and only if $\{\rho < \varepsilon\}$ is a neighbourhood of $1$ for all $\varepsilon > 0$.
In particular, if $\rho' \leq \rho$ are semi-norms and $\rho$ is continuous, then so is $\rho'$.

Whenever a group $G$ acts on a metric space $X$ by isometry, every $x \in X$ gives rise to a semi-norm $\rho_x(g) = d(x,gx)$.
We encounter semi-norms of this form regularly.
For example, we can restate \autoref{fct:ContinuousAction} as:
\begin{fct}
  Let $\varphi \in L^1(G)$.
  Then $\rho_\varphi(g) = \|g\varphi - \varphi\|_1$ is a continuous semi-norm.
\end{fct}
Similarly, if $E$ is a unitary representation and $\xi \in E$, then $\xi \in E^{\mathrm{c}}$ if and only if $\rho_\xi(g) = \|g\xi - \xi\|$ is continuous.

\begin{dfn}
  \label{dfn:EquicontinuousUltraproduct}
  Let $(E_i : i \in I)$ be continuous unitary representations of $E$.
  We define their \emph{equicontinuous ultraproduct}, denoted $\prodec_\cU E_i$, to consist of all $\xi = [\xi_i] \in \prodB_\cU E_i$ such that for some continuous semi-norm $\rho$, we have $\rho_{\xi_i} \leq \rho$ for all (or equivalently, $\cU$-many) $i$, as well as the limits of such:
  \begin{gather*}
    \prodec_\cU E_i = \overline{\left\{ [\xi_i] \in \prodB_\cU E_i : \rho_{\xi_i} \leq \rho \ \text{for some common continuous semi-norm} \ \rho \right\}}.
  \end{gather*}
\end{dfn}

It is clear that $\prodec_\cU E_i$ is a $G$-invariant Hilbert space, and if $\rho_{\xi_i} \leq \rho$ for some continuous $\rho$ and all $i$, then $\rho_\xi \leq \rho$ as well.
Therefore:
\begin{gather*}
  \prodec_\cU E_i \subseteq \prodc_\cU E_i \subseteq \prodB_\cU E_i.
\end{gather*}

% \begin{quote}
%   Digression: can we have $\prodec_\cU E_i \subsetneq \prodc_\cU E_i$?
%   Assume that
%   \begin{quote}
%     there exists a continuous semi-norm $\rho$ and $\varepsilon > 0$, such that for every neighbourhood $U$ of $1$ and every finite subset $F \subseteq G$, there exists $E$ and $\xi \in E$ such that
%     \begin{itemize}
%     \item $\|\xi\| = 1$
%     \item $\|g \xi - \xi\| \leq \rho(g)$ for all $g \in F$.
%     \item $\|g \xi - \xi\| \geq \varepsilon$ for some $g \in U$.
%     \end{itemize}
%   \end{quote}
%   In this case, we have strict inequality (but the ultrafilter is on an uncountable set).
%   Indeed, let $I$ be the set of such pairs.
% \end{quote}

\begin{prp}
  \label{prp:Ultraproducts}
  Assume that $G$ is locally compact.
  Let $(E_i,\pi_i)$ be continuous representations for $i \in I$, and let $E = \prodB_\cU E_i$.
  Each $\pi_i$ extends to a $*$-morphism $\pi_i\colon M(G) \rightarrow B(E_i)$, giving rise to $\pi \colon M(G) \rightarrow \prodB_\cU B(E_i) \subseteq B(E)$.
  In addition, since $\prodc_\cU E_i = E^c$ is a continuous representation, it gives rise to a $*$-morphism $\sigma \colon M(G) \rightarrow B(E^c)$.
  Then the following are equivalent for $\xi \in E$:
  \begin{enumerate}
  \item $\xi \in E^{\mathrm{c}}$ and $\pi(\varphi) \xi = \sigma(\varphi) \xi$ for all $\varphi \in L^1(G)$.
  \item $\xi$ belongs to the non-degenerate part of the $*$-representation $\pi\colon L^1(G) \rightarrow B(E)$ (i.e., to the non-degenerate ultraproduct in the sense of \cite{Cherix-Cowling-Straub:FilterProducts}).
  \item $\xi \in \prodec_\cU E_i$.
  \end{enumerate}
\end{prp}
\begin{proof}
  \begin{cycprf}
  \item Let $(e_\alpha)$ be any approximate identity of norm one in $L^1(G)$.
    Since $\sigma\colon G \rightarrow U(E^{\mathrm{c}})$ is continuous, the $*$-representation $\sigma\colon L^1(G) \rightarrow B(E^{\mathrm{c}})$ is non-degenerate.
    Therefore $\pi(e_\alpha) \xi = \sigma(e_\alpha) \xi \rightarrow \xi$, so $\xi$ is in the non-degenerate part of $\pi\colon L^1(G) \rightarrow B(E)$.
  \item Assume that $\xi = \pi(\varphi) \zeta$ for $\varphi \in L^1(G)$ and $\zeta = [\zeta_i]$ in $E$.
    We may assume that $\|\zeta_i\| \leq \|\zeta\|$ for all $i$, and that $\xi = [\xi_i]$, where $\xi_i = \pi_i(\varphi) \zeta_i$.
    Then $\rho_{\xi_i} < \|\zeta_i\| \rho_\varphi \leq \|\zeta\| \rho_\varphi$, and $\xi \in \prodec_\cU E_i$.
    Since $\prodec_\cU E_i$ is complete, this is enough.
  \item[\impfirst]
    Finally, let $\xi \in \prodec_\cU E_i \subseteq \prodc_\cU E_i$, and let $\varphi \in L^1(G)$.
    We need to show that $\pi(\varphi) \xi = \sigma(\varphi) \xi$.
    The latter is a closed condition in both $\xi$ and $\varphi$.
    We may therefore assume that $\xi = [\xi_i]$, $\rho$ is a continuous semi-norm, $\rho_{\xi_i} \leq \rho$ for all $i$ (and therefore $\rho_\xi \leq \rho$), and that $\varphi$ has compact support.
    Let $\varepsilon > 0$.
    By a partition of unity argument, we can express $\varphi$ as a finite sum $\sum_{m<n} \varphi_m$, where the $\varphi_m$ are supported on disjoint sets, each contained in a single translate $g_m \{\rho < \varepsilon\}$.
    Then, by Lemma \ref{lem:ContinuousRepresentationMG}, we have
    \begin{gather*}
      \left\| \sigma(\varphi) \xi - \sum_{m<n} \|\varphi_m\| g_m \xi \right\|
      \leq \sum_{m<n} \Bigl\| \sigma(\varphi_m) \xi - \|\varphi_m\| g_m \xi \Bigr\|
      \leq \sum_{m<n} \varepsilon \|\varphi_m\|
      = \varepsilon \|\varphi\|.
    \end{gather*}
    By the same argument, for all $i$ we have
    \begin{gather*}
      \left\| \pi_i(\varphi) \xi_i - \sum_{m<n} \|\varphi_m\| g_m \xi_i \right\|
      \leq \varepsilon \|\varphi\|,
    \end{gather*}
    and therefore, in the ultraproduct,
    \begin{gather*}
      \left\| \pi(\varphi) \xi - \sum_{m<n} \|\varphi_m\| g_m \xi \right\|
      \leq \varepsilon \|\varphi\|.
    \end{gather*}
    It follows that $\|\sigma(\varphi) \xi - \pi(\varphi) \xi\| \leq 2 \varepsilon \|\varphi\|$, and since $\varepsilon$ was arbitrary, $\sigma(\varphi) \xi = \pi(\varphi) \xi$.
  \end{cycprf}
\end{proof}

This means in particular that the non-degenerate ultraproduct of $*$-representations of $L^1(G)$ and the equicontinuous ultraproduct of representations of $G$ are, in essence, the same construction.

\begin{qst}
  Find an example where $\prodec_\cU E_i \subsetneq \prodc_\cU E_i$ (or show that this can never happen).
\end{qst}

\section{Property (T)}
\label{sec:PropertyT}

Suppose that $G$ is a topological group, $Q\subseteq G$, and $\varepsilon > 0$.
Let $E$ be a unitary representation of $G$.

A vector $\xi\in E$ is \emph{$(Q,\varepsilon)$-invariant} if $\sup_{g\in Q}\|g\xi-\xi\| < \varepsilon\|\xi\|$.
Note, in particular, that a $(Q,\varepsilon)$-invariant vector must be nonzero.
We say that $\xi\in E$ is \emph{$G$-invariant} if $g\xi=\xi$ for all $g\in G$.
We say that $(Q,\varepsilon)$ is a \emph{Kazhdan pair} for $G$ if, whenever $E$ is a unitary representation of $G$ with a $(Q,\varepsilon)$-invariant vector, then $E$ has a nonzero $G$-invariant vector.
If there is $\varepsilon>0$ such that $(Q,\varepsilon)$ is a Kazhdan pair for $G$, then we say that $Q$ is a \emph{Kazhdan set} for $G$.
Finally, $G$ is said to have \emph{property (T)} if it has a compact Kazhdan set.

\begin{fct}
  \label{fct:LocallyCompactPropertyT}
  A locally compact group $G$ with property (T) is compactly generated.
  Moreover, if $Q \subseteq G$ has non-empty interior, then it is generating if and only if it is a Kazhdan set.
\end{fct}
\begin{proof}
  See \cite[Theorem~1.3.1 and Proposition~1.3.2]{Bekka-DeLaHarpe-Valette:KazhdanT}.
\end{proof}

For each $\varphi\in L^1(G)$, we let
\begin{gather*}
  \Fix_\varphi = \{\xi\in S_\varphi : g\xi = \xi \ \text{for all} \ g \in G\}.
\end{gather*}
For $\varphi \in L^1(G)$ and $m \in \bN$, define an open neighbourhood of the identity by
\begin{gather*}
  U_{\varphi,m} = \left\{ g\in G \ : \ \|\varphi-g\varphi\|_1 < 2^{-m} \right\}.
\end{gather*}
For $K \subseteq G$ compact, choose $K_{\varphi,m} \subseteq K$ finite so that $K \subseteq K_{\varphi,m} U_{\varphi,m}$.

Consider continuous unitary representations of $G$, as models of $T^{L^1(G)}$, in the language $\cL^{L^1(G)}$, as described in \autoref{sec:NonDegenerateStarRepresentations}.
Let
\begin{gather*}
  \alpha_{K,\varphi,m}(x) = \max_{g\in K_{\varphi,m}} d(x,gx),
  \qquad
  \Phi_{K,\varphi}(x) = \sup_m 2^{-m}\alpha_{K,\varphi,m}(x).
\end{gather*}
Both are definable predicates in a single variable $x$ of sort $S_\varphi$.

If $E$ is such a representations and $\xi \in \Fix_\varphi$, then $\Phi_{K,\varphi}(\xi) = 0$.
Conversely, if $\xi \in S_\varphi$ and $\Phi_{K,\varphi}(\xi) = 0$, then $g \xi = \xi$ for all $g \in K$.
In particular, if $K$ is a compact generating set, then
\begin{gather*}
  \xi \in \Fix_\varphi \qquad \Longleftrightarrow \qquad \Phi_{K,\varphi}(\xi) = 0.
\end{gather*}

\begin{thm}
  \label{thm:PropertyT}
  Suppose that $G$ is a compactly generated locally compact group.
  Then $G$ has property (T) if and only if, for each $\varphi\in L^1(G)$, we have that $\Fix_\varphi$ is a definable subset of $S_\varphi$ in the sense of the theory $T^{L^1(G)}$.
\end{thm}
\begin{proof}
  First suppose that $G$ has property (T).
  Let $K$ be a compact generating set for $G$, and let $\varphi \in L^1(G)$ be arbitrary.
  Since the result is trivial for $\varphi=0$, we assume that $\varphi \neq 0$.
  By \autoref{fct:LocallyCompactPropertyT}, $K$ is a Kazhdan set for $G$.

  Let $\varepsilon>0$ be such that $(K,\varepsilon)$ is a Kazhdan pair for $G$.
  We already know that $\Fix_\varphi$ is the zero-set of $\Phi_{K,\varphi}$, and we need to show that if $\Phi_{K,\varphi}(\xi)$ is small, then $\xi$ is close to $\Fix_\varphi$.
  Indeed, assume that $\Phi_{K,\varphi}(\xi) < 2^{-2m-2}$.
  Then $\alpha_{K,\varphi,m+1}(\xi) < 2^{-m-1}$.
  Let $f \in K$.
  Then $f = gh$, where $g \in K_{\varphi,m+1}$ and $h \in U_{\varphi,m+1}$.
  The former implies that $\|\xi - g\xi\| < 2^{-m-1}$, and the latter that $\|\xi - h\xi\| < 2^{-m-1}$.
  Since the action of $g$ is isometric,
  \begin{gather*}
    \|\xi - f \xi\|
    \leq \| \xi - g\xi\| + \| g\xi - gh \xi \|
    \leq \| \xi - g\xi\| + \| \xi - h \xi \|
    < 2^{-m}.
  \end{gather*}
  In other words, $\xi$ is $(K,2^{-m})$-invariant.
  By \cite[Proposition~1.1.9]{Bekka-DeLaHarpe-Valette:KazhdanT} there exists a $G$-invariant vector $\zeta \in E$ such that $\|\xi-\zeta\| \leq \frac{\|\xi\|}{2^m \varepsilon}$ and $\|\zeta\| \leq \|\xi\| \leq \|\varphi\|_1$.
  The latter, together with the fact that $\zeta$ is fixed, implies that $\zeta \in S_\varphi$, so $\zeta \in \Fix_\varphi$.
  Thus indeed, if $\Phi_{K,\varphi}(\xi)$ is sufficiently small, then $\xi$ is as close as desired to $\Fix_\varphi$.

  For the converse implication, let us fix a compact generating $K$ for $G$, with non-empty interior.
  Set $\varphi = \frac{1_K}{H(K)} \in L^1(G)$.
  Then by hypothesis, $\Fix_\varphi$ is definable in $S_\varphi$, and we shall prove that $(K,\varepsilon)$ is a Kazhdan pair for $G$ for $\varepsilon > 0$ small enough.
  Let $\mu\in M(G)$ be such that $d\mu=\varphi dH$.
  This is a probability measure on $G$ concentrated on $K$.
  Since $\Fix_\varphi$ is a definable set and equals the zeroset of $\Phi_\varphi$, there is $\delta > 0$ such that, if $\Phi_\varphi(\xi) < \delta$, then $d(\xi,\Fix_\varphi)<\frac{1}{2}$.
  Take $\varepsilon > 0$ small enough so that if $\xi\in S_\varphi$ is $(K,3\varepsilon)$-invariant, then $\Phi_\varphi(\xi) < \delta$.
  Now suppose that $\xi$ is a $(K,\varepsilon)$-invariant unit vector, and let $\hat{\xi} = \pi(\varphi)\xi$.
  By \autoref{lem:ContinuousRepresentationMG}, $\|\hat{\xi}-\xi\| \leq \varepsilon$, so $\hat{\xi}$ is a $(K,3\varepsilon)$-invariant vector.
  Since $\hat{\xi} \in S_\varphi$, it follows that $\Phi_\varphi(\hat{\xi}) < \delta$, whence $d(\hat{\xi},\Fix_\varphi) < \frac{1}{2}$.
  If $\zeta\in \Fix_\varphi$ is such that $d(\hat{\xi},\zeta)<\frac{1}{2}$, then $d(\xi,\zeta) < \varepsilon + \frac{1}{2}$.
  We may assume that $\varepsilon <  \frac{1}{2}$, so $\zeta\not=0$.
  We have thus found a nonzero $G$-invariant vector, which shows that $(K,\varepsilon)$ is a Kazhdan pair for $G$.
\end{proof}

\bibliographystyle{begnac}
\bibliography{begnac}

\end{document}